\documentclass[11pt, usletter, reqno, twoside, makeidx]{amsart}
\usepackage[margin=2.8cm, marginpar=2.3cm]{geometry}

\usepackage[color=blue!20!white,textsize=tiny]{todonotes}
\usepackage{bm}
\usepackage{multicol}
\usepackage{bm,amsmath,amsthm,amssymb,mathtools,verbatim,amsfonts,tikz-cd,mathrsfs}
\usepackage{mathrsfs}
\usepackage{amsmath} 
\usepackage{amsthm}
\usepackage{amssymb} 
\usepackage{amsfonts}
\usepackage{tikz-cd}
\usepackage{graphicx}
\usepackage{xcolor}
\definecolor{darkgreen}{rgb}{0,0.5,0}
\usepackage[normalem]{ulem}
\usepackage{comment}
\usepackage{caption,subcaption,pinlabel}
\usepackage{graphics}
\usepackage{hyperref}
\usepackage{scalerel}
\usepackage{comment}
\usepackage{enumitem}
\usepackage{mathtools}
\usepackage{tikz}
\usetikzlibrary{matrix,arrows,decorations}
\usepackage{mathabx}
\usetikzlibrary{matrix}
\usepackage{scrextend}
\usepackage[utf8]{inputenc}
\newtheorem{theorem}{Theorem}[section]
\newtheorem{lemma}[theorem]{Lemma}

\newtheorem{corollary}[theorem]{Corollary}

\theoremstyle{definition}
\newtheorem{definition}[theorem]{Definition}
\newtheorem{example}[theorem]{Example}

\newtheorem{remark}[theorem]{Remark}

\newcommand{\bunderline}[1]{\underline{#1\mkern-2mu}\mkern2mu }

\def\du {\bar{d}}
\def\dl {\bunderline{d}}

\def\spinc {{\operatorname{spin^c}}}

\def\CF {\mathit{CF}}

\newcommand \CFm {\CF^-}

\let\int\relax
\newcommand{\int}{\mathring}

\usepackage{accents}

\DeclareMathSymbol{\wtilde}{\mathord}{largesymbols}{"65}

\mathchardef\mhyphen="2D

\newcommand{\gr}{\text{gr}}

\newcommand{\Vl}{\underline{V}_0}
\newcommand{\Vu}{\overline{V}_0}

\title{A note on cables and the involutive concordance invariants}

\author{Kristen Hendricks}
\address{Department of Mathematics\\Rutgers University\\  New Brunswick, NJ 08854}
\email{kristen.hendricks@rutgers.edu}
\thanks{KH was partially supported by NSF CAREER grant DMS-2019396.}

\author{Abhishek Mallick}
\address{Department of Mathematics\\Rutgers University\\  New Brunswick, NJ 08854}
\email{abhishek.mallick@rutgers.edu}
\thanks{AM was partially supported by NSF CAREER grant DMS-2019396.}


\begin{document}
\begin{abstract}
 We prove a formula for the involutive concordance invariants of the cabled knots in terms of those of the companion knot and the pattern knot. As a consequence, we show that any iterated cable of a knot with parameters of the form (odd,1) is not smoothly slice as long as either of the involutive concordance invariants of the knot is nonzero. Our formula also gives new bounds for the unknotting number of a cabled knot, which are sometimes stronger than other known bounds coming from knot Floer homology.
\end{abstract}

\maketitle

\section{Introduction}
Cabling is a natural operation on a knot which acts on the smooth concordance group. There has been considerable interest in characterizing the behavior of various knot concordance invariants under cabling. Typically, one hopes to prove a formula which relates the values of some concordance invariant of a cabled knot to the value of the invariant on the companion knot and the value of the invariant on the pattern knot, in this case a torus knot. Some examples of invariants which are known to admit such formulas include the Levine-Tristram signatures \cite{Litherland:torus}, the Heegaard Floer $\tau$-invariant and $\epsilon$-invariant  \cite{Hedden_cabling, Petkova_cabling, Hom_cabling}, Rasmussen's $V_0$-invariant \cite{Hom_Wu}, and the Heegaard Floer $\nu^{+}$-invariant \cite{Wu_cabling}, among others. 

Recently, concordance invariants stemming from the involutive variant of Heegaard Floer homology defined by the first author and Manolescu \cite{HM} have been shown to be fruitful in many applications. In this paper we will be interested in the concordance invariants $\Vl(K)$ and $\Vu(K)$, which can be thought of as the involutive analog of Rasmussen's $V_0(K)$ invariant \cite[Subsection 6.7]{HM}. An early indication of the utility of these invariants was that the involutive concordance invariants are able to detect non-sliceness of certain rationally slice knots. Later, equivariant refinements of these invariants defined by the second author, Dai, and Stoffregen \cite{DMS, mallick2022knot} were used in  to show that $(2,1)$-cable of the figure-eight knot is not smoothly slice \cite{dai20222}.

In this article we give the cabling formula for the involutive concordance invariants. Let $K_{p,q}$ denote the $(p,q)$-cable of a knot $K$, and $V_s$ represent the generalization \cite{ni2015cosmetic} of Rasmussen's $V_0$-invariant \cite[Definition 7.1]{Rasmussen:V0}.

\begin{theorem}\label{thm: intro_formula}
Let $p$ and $q$ be positive integers with $(p,q)=1$. Then the involutive concordance invariants satisfy the following relations:
\begin{enumerate}

\item If $p$ is odd, we have
 \[
\Vl(K_{p,q})= \Vl(K) + V_0(T_{p,q}) , \; \; \Vu(K_{p,q})= \Vu(K) + V_0(T_{p,q}).
 \]
 
 \item If $p$ is even,
 \[
\Vl(K_{p,q})= \mathrm{max}\{ V_{\left \lfloor \frac{s}{p} \right \rfloor }(K), V_{\left \lfloor \frac{p+q -1-s}{p} \right \rfloor }(K) \} + V_0(T_{p,q})  , \; \; \Vu(K_{p,q})= V_0(T_{p.q}),
 \]
 \end{enumerate}
where $s \equiv \frac{p+q-1}{2} (\mathrm{mod} \; q) $ and $0\leq s \leq q-1$.

\end{theorem}
\noindent
Theorem~\ref{thm: intro_formula} implies that the involutive concordance invariants are determined by the knot Floer homology  when the longitudinal winding parameter $p$ is even. On the other hand, when $p$ is odd, the involutive concordance invariants incorporate the corresponding invariants for the companion knot. 

Theorem~\ref{thm: intro_formula} implies the following corollary. Let $K_{p_1,q_1; p_2, q_2; p_3, q_3; \cdots ; p_k,q_k}$
represent the iterated cable of the knot $K$. For example, $K_{p_1,q_1; p_2, q_2}$ is the $(p_2,q_2)$-cable of the $(p_{1},q_{1})$-cable of $K$.
\begin{corollary}\label{cor:intro_ite_cable} If the parameters $p_i$ are all odd and positive, and either $\Vu(K)$ or $\Vl(K)$ is non-zero, then $K_{p_1,1; p_2,1; p_3,1; \cdots ; p_k,1}$ is not smoothly slice.
\end{corollary}
Corollary~\ref{cor:intro_ite_cable} is related to a famous open question by Miyazaki \cite[Question 3]{Miyazaki} which asks whether there are non-slice knots $K$ for which the cable knot $K_{p,1}$ is slice.

\begin{remark} It is possible to produce examples of topologically (and therefore algebraically) slice knots which are trivial with respect to ordinary Heegaard Floer homology and involutively nontrivial, for which the corollary above detects nonsliceness of the cable; for example, one such family appears in \cite[Theorem 1.1]{HKP:slice}. The authors are, however, not presently aware of a family of examples which are not better elucidated by other methods.
\end{remark}

We now explain the strategy for the proof of Theorem~\ref{thm: intro_formula}. A direct approach would be to determine the full involutive knot Floer invariant of the cabled knot, which takes the form of the knot Floer chain complex of Ozsv{\'a}th-Szab{\'o} and Rasmussen $\mathcal{CKF(K)}$ together with the knot conjugation symmetry $\iota_K$ on the complex. The involutive concordance invariants of the knot may then in principle be extracted from this data. However, it is challenging to compute the full knot Floer chain complex for a general cabled knot, and additionally challenging to determine the symmetry $\iota_K$. In general, direct computations of $\iota_K$ have only been carried out for simple chain complexes in which the map is determined by its algebraic properties, and for tensor products of those complexes.

Instead, our proof takes an indirect approach, using certain consequences of the surgery formula in involutive Heegaard Floer homology proved by Hom-Stoffregen-Zemke and the first author \cite{hendricks2020surgery} to deduce the formulas presented in Thoerem~\ref{thm: intro_formula}. This approach is similar to one used by Ni-Wu to compute the Heegaard Floer correction terms of three-manifolds using surgery formulas \cite{ni2015cosmetic}. 

\subsection{Applications to the unknotting number}
We now discuss an application of Theorem~\ref{thm: intro_formula} to the unknotting number. Recall that the unknotting number $u(K)$ of a knot $K$ is the minimum number of times the knot must be passed through itself to turn it into an unknot. There are many well-known bounds in the literature for the unknotting number of a knot. In one recent example, in \cite{AE:unknotting}, Alishahi and Eftekhary gave a bound using the torsion order of the $HFK^{-}$ flavor of the knot Floer homology. This bound was later used by Hom, Lidman, and Park \cite{HLP:cables} to give bounds for the unknotting number of cabled knots. In particular, the authors showed that
\[
u(K_{p,q}) \geq p.
\]
It is also well-known that the unknotting number is greater than or equal to the slice genus of the knot. Hence using the relationship of Rasmussen's $V_0$ invariant to the slice genus \cite{Rasmussen:V0}, together with the cabling formula for the $V_0$-invariant \cite{Hom_Wu}, one can also write down the bounds
\[
u(K_{p,q}) \geq 2 V_0(K) + 2 V_0(T_{p,q}), \; \text{if} \; g_4(K_{p,q}) \; \text{is even}\]
and
\[u(K_{p,q}) \geq 2 V_0(K) + 2 V_0(T_{p,q}) -1, \; \text{if} \; g_4(K_{p,q}) \; \text{is odd}.
\]
In particular, without knowing the slice genus of the knot, we have that
\[
u(K_{p,q}) \geq 2 V_0(K) + 2 V_0(T_{p,q}) -1.
\]

Using Theorem~\ref{thm: intro_formula}, we improve the above bounds:
\begin{theorem}\label{thm: intro_unknotting}
Let $K_{p,q}$ be a cabled knot with $p$ odd. Then
\[u(K_{p,q}) \geq 2\Vl(K) + 2V_0(T_{p,q}) - 2 \text{ and } u(K_{p,q}) \geq -2\Vu(K) - 2V_0(T_{p,q}) - 2. \]
\end{theorem}
More precisely, one has

\begin{enumerate}
\item $u(K_{p,q}) \geq 2\Vl(K) + 2V_0(T_{p.q}) - 1 \text{ and } u(K_{p,q}) \geq -2\Vu(K) - 2V_0(T_{p,q}) - 1, \; \; g_4(K_{p,q}) \; \text{odd}$,
\item $u(K_{p,q}) \geq 2\Vl(K) + 2V_0(T_{p,q}) - 2 \text{ and } u(K_{p,q}) \geq -2\Vu(K) - 2V_0(T_{p,q}) - 2  ,\;\; g_4(K_{p,q}) \; \text{even}$
\end{enumerate}

\noindent although of course one in general one would like to apply Theorem~\ref{thm: intro_unknotting} in situations where the slice genus is unknown.

\begin{remark} A statement similar to Theorem~\ref{thm: intro_unknotting} could also be made for the 4-dimensional positive clasp number of a knot in place of the unknotting number. In another direction, because $\Vl$ and $\Vu$ are concordance invariants, the bound from Theorem~\ref{thm: intro_unknotting} is in fact a bound on the concordance unknotting number of $K_{p,q}$, which is to say the minimum unknotting number of a knot $J$ concordant to $K_{p,q}$.\end{remark}

\noindent

In Example \ref{example:beats-previous} we show that there are infinitely many knots for which the bound from Theorem~\ref{thm: intro_unknotting} is stronger than that from \cite{HLP:cables} or the aforementioned bound from the $V_0$-invariant. 
\subsection{Organization}
This paper is organized as follows. In Section~\ref{sec:background} we review the definition of the involutive concordance invariants and structural features of the involutive surgery formula. We then prove Theorem \ref{thm: intro_formula}, Corollary \ref{cor:intro_ite_cable}, and Theorem \ref{thm: intro_unknotting} in Section~\ref{sec:proofs}. Finally, in Section~\ref{sec:examples} we consider an example in which our formula improves on other bounds from knot Floer homology.

\subsection{Acknowledgments}
We thank Jen Hom, Sungkyung Kang, Tye Lidman, JungHwan Park, and Ian Zemke for helpful comments. Portions of this work were carried out while the second author was in attendance at a Simons Semester on ``Knots, Homologies and Physics'' at the Institute of Mathematics of the Polish Academy of Science; other portions were carried out while the first author was present at the BIRS workshop ``What's your Trick?: A nontraditional conference in low-dimensional topology.'' We are grateful to both institutions for their hospitality. Finally, we thank the anonymous referee for helpful comments.

\section{Background on involutive Heegaard Floer homology} \label{sec:background}
In this section, we briefly recall the definition of the involutive knot concordance invariants of a knot and the involutive correction terms of a three-manifold from \cite{HM} and their relationship to surgeries on the knot \cite{hendricks2020surgery}. We begin by recalling the algebraic setup.

\begin{definition} \label{def:iota_complex} An $\emph{iota-complex}$ is a pair $(C,\iota)$ of the following form:
\begin{itemize}
\item $C$ is an absolutely $\mathbb Q$-graded and relatively $\mathbb{Z}$-graded finitely-generated chain complex over $\mathbb{F}_2[U]$ with the property that $U^{-1}H_*(C) \simeq \mathbb{F}_2[U,U^{-1}]$ with some grading shift.
\item $\iota$ is a grading-perserving chain map such that $\iota^2$ is chain homotopic to the identity.
\end{itemize}
\end{definition}

To a rational homology sphere $Y$ together with a $\spinc$ structure $\mathfrak s$ on $Y$, Ozsv{\'a}th and Szab{\'o}'s Heegaard Floer homology associates a chain complex $\CFm(Y,\mathfrak s)$ \cite{OS:holomorphicdisks, OS:properties}; if $\mathfrak{s}$ is conjugation-invariant, involutive Floer homology appends to this data a chain map $\iota$. The pair $(\CFm(Y,\mathfrak s), \iota)$ is an iota-complex in the sense of Definition~\ref{def:iota_complex}.

The involutive Heegaard Floer invariants satisfy a tidy connected sum formula, as follows. 
\begin{equation}\label{eqn:connect_sum}
\CFm(Y_1 \# Y_2, \mathfrak s_1 \# \mathfrak s_2, \iota) \simeq (\CFm(Y_1, \mathfrak s_1) \otimes \CFm(Y_2, \mathfrak s_2), \iota_1 \otimes \iota_2).
\end{equation}
That is, the two complexes are chain homotopy equivalent via maps which commute up to homotopy with the involutions in their pairs. This relationship is sometimes called a \emph{strong equivalence}. The statement for the chain complexes is due to \cite[Section 6]{OS:properties} and the statement for the involution is due to \cite[Theorem 1.1]{HMZ}.

In ordinary Heegaard Floer homology, the \emph{$d$-invariant} of a chain complex $C$ satisfying the criterion above is the maximum grading of an element $a$ in $C$ such that $[U^n a] \neq 0$ for all positive $n$, and the $d$-invariant or \emph{correction term} of $(Y,\mathfrak s)$ is $d(\CFm(Y,s))$ \cite{ozsvath2003absolutely}. From this, one may extract Rasmussen's concordance invariant $V_0$ as $V_0(K) = -\frac{1}{2} d(\CFm(S^3_{+1}(K))$ \cite{ozsvath2008knot, Rasmussen:V0, Peters}. The $d$-invariant is additive under connected sum; that is, 
\begin{equation} \label{eqn:d-additive}
d((Y_1\#Y_2), \mathfrak s_1\# \mathfrak s_2) = d(Y_1, \mathfrak s_1) + d(Y_2, \mathfrak s_2).
\end{equation}

An important special case is that of the lens space $L(p,q)$. The iota-complex associated to $L(p,q)$ in any $\spinc$ structure $\mathfrak s$ is a copy of the pair $(\mathbb F_2[U], \mathrm{Id})$, with the grading of the element $1$ in $\mathbb F_2[U]$ being $d(L(p,q),\mathfrak s)$. Therefore, taking the connected sum with a lens space has the sole effect of imposing a grading shift on $(\CFm(Y), \mathfrak s)$.

We may now define the involutive variants of the $d$-invariant. We follow \cite[Lemma 2.12]{HMZ}, which is a reformulation of the original definitions from \cite{HM}.

\begin{definition}
Let $(C,\iota)$ be an iota-complex. Then $\dl(C,\iota)$ is the maximum grading of a homogeneous cycle $a \in C$ such that $[U^{n}a] \neq 0$ for all positive $n$ and furthermore there exists an element $b$ such that $\partial b= (\mathrm{id} + \iota)a$.
\end{definition}

\begin{definition}
Let $(C,\iota)$ be an iota-complex. Consider triples $(x,y,z)$ consisting of elements of $C$, with at least one of $x$ or $y$ nonzero such that $\partial y = (\mathrm{id} + \iota)x$, $\partial z = U^{m} x$, for some $m \geq 0$ and $[U^n(U^{m}y + (\mathrm{id} + \iota)z)] \neq 0$ for all $n \geq 0$. If $x \neq 0$, assign this triple the value $\gr(x)+1$; if $x=0$, assign this triple the value $\gr(y)$. Then $\du(C,\iota)$ is defined as the maximum of these grading values across all valid triples $(x,y,z)$.
\end{definition}

Given a three-manifold $Y$ with conjugation-invariant $\spinc$-structure $\mathfrak s$, we say that $\dl(Y, \mathfrak s) = \dl(\CFm(Y,\mathfrak s))$ and likewise $\du(Y, \mathfrak s) = \du(\CFm(Y, \mathfrak s)).$ The invariants $\dl$ and $\du$ are not additive under connected sum, but instead, for $\mathfrak s = \mathfrak s_1 \# \mathfrak s_2$, satisfy the following \cite[Proposition 1.3]{HMZ}:
\begin{equation} \label{eqn:not-additive-but}\dl(Y_1,\mathfrak s_1) + \dl(Y_2, \mathfrak s_2) \leq \dl(Y_1\#Y_2,\mathfrak s) \leq \dl(Y_1,\mathfrak s_1) + \du(Y_2, \mathfrak s_2) \leq \du(Y_1\#Y_2,\mathfrak s) \leq \du(Y_1,\mathfrak s_1) + \du(Y_2, \mathfrak s_2).
\end{equation}

We may now define the invariants $\Vl$ and $\Vu$.

\begin{definition}
Let $K$ be a knot in $S^3$. The involutive concordance invariants of $K$ are
\[
\Vl(K):= -\frac{1}{2}\dl(S^3_{+1}(K)), \qquad \Vu(K):= -\frac{1}{2}\du(S^3_{+1}(K)).
\]
\end{definition} 
These invariants also admit a more general relationship to surgeries, which is an extension of a formula for the non-involutive d-invariant proved by Ni and Wu \cite{ni2015cosmetic}. The set of $\spinc$-structures on $p/q$-surgery on $K$ may be identified with $\mathbb{Z}/p\mathbb Z$; in this paper we follow the convention used by \cite{OS11} and \cite{ni2015cosmetic} for this identification. Ni and Wu show that, given $0\leq s\leq p-1$
\begin{equation}\label{eq:NiWu}
d(S^3_{p/q}(K), [s])= d(L(p,q),[s]) - 2\mathrm{max}\{ V_{\left \lfloor \frac{s}{q} \right \rfloor }(K), V_{\left \lfloor \frac{p+q -1-s}{q} \right \rfloor }(K) \}.
\end{equation}
Here $V_i$ represents the $i$th concordance invariant in the sequence defined in \cite{OS11, Rasmussen:V0}, which sequence generalizes $V_0$. In \cite{hendricks2020surgery}, Hom, Stoffregen, and Zemke and the first author prove a surgery formula in involutive Heegaard Floer homology, one of whose consequences is an involutive analog of the relationship above, which we now recall. With respect to the convention above, the conjugation action on $\spinc$-structures is
\begin{equation}\label{equ:spinc_action}
J([i])=[p+q-1-i], \; \text{where} \; [i] \in \mathbb{Z}_p.
\end{equation}
In particular, if $p$ and $q$ are both odd, then $\left[\frac{q-1}{2}\right] \in \mathbb{Z}_p$ represents the unique self-conjugate $\spinc$-structure on the manifold $S^{3}_{p/q}(K)$. When $p$ is even and $q$ is odd, there are two self-conjugate $\spinc$-structures on the manifold $S^{3}_{p/q}(K)$, to wit $\left[\frac{q-1}{2}\right]$ and $\left[\frac{p+q-1}{2}\right]$. Finally, when $p$ is odd and $q$ is even, there is a unique $\spinc$-structure $\left[\frac{p+q-1}{2}\right]$. With this in mind, we have the following.
\begin{theorem}\cite[Proposition 1.7]{hendricks2020surgery} \label{thm:surgery_formula}
Suppose that $q$ is odd, then we have
\[
\dl\left(S^3_{p/q}(K), \left[ \frac{q-1}{2}\right]\right)= d\left(L(p,q),\left[\frac{q-1}{2}\right]\right) - 2\Vl(K)\] \[\du\left(S^3_{p/q}(K), \left[\frac{q-1}{2}\right]\right)= d\left(L(p,q),\left[ \frac{q-1}{2}\right]\right) - 2\Vu(K).
\] 
If one of $p$ or $q$ is even, then
\[
\dl\left(S^3_{p/q}(K), \Big[\frac{p+q-1}{2} \Big]\right)= d\left(S^3_{p/q}(K),\Big[\frac{p+q-1}{2} \Big]\right)\] \[\du\left(S^3_{p/q}(K), \Big[\frac{p+q-1}{2} \Big]\right)= d\left(L(p,q),\Big[\frac{p+q-1}{2} \Big]\right).
\]
\end{theorem}

Note that the original version of the theorem is stated using a different indexing convention for the $\spinc$ structures, which we have translated into the Ozsv{\'a}th-Szab{\'o} convention here; for more on this, see \cite[Pages 204-205]{hendricks2020surgery}. For the statement of the full surgery formula of which this relationship is a consequence, see \cite[Sections 1.2 and 1.3]{hendricks2020surgery}.

\section{Proofs of the main Theorems} \label{sec:proofs}
\subsection{A cabling formula for the involutive concordance invariants}
In this subsection, we prove Theorems~\ref{thm: intro_formula} and ~\ref{thm: intro_unknotting}. A key ingredient of our proof will be the following classical diffeomorphism (see \cite{Moser:torus} for a first case and \cite[Corollary 7.3]{Gordon:satellite} for the general statement; a summary of the proof can be found in in \cite[Section 2.4]{Hedden_cabling_two}):
\[
S^3_{pq}(K_{p,q}) = S^3_{q/p}(K) \# L(p,q).
\]
Let us now denote the projections of a  $\spinc$-structure to the two summands $ S^3_{q/p}(K)$ and $L(p,q)$ as $\pi_1$ and $\pi_2$ respectively. We may treat these as functions
\[
\pi_1 \colon \mathbb{Z}_{pq} \rightarrow \mathbb{Z}_q, \qquad \pi_2 \colon \mathbb{Z}_{pq } \rightarrow \mathbb{Z}_p.
\]

We record the following lemma, which is a direct consequence of \cite[Lemma 4.1]{Wu_cabling}.

\begin{lemma}\label{lem: spinc_decom}
The self-conjugate $\spinc$-structure $\left[0\right] \in \mathbb{Z}_{pq}$ projects to the summands as follows:
\begin{enumerate}

\item If both $p$ and $q$ are odd, we have
 \[
  \pi_1(0) \equiv \frac{p-1}{2} \; (\mathrm{mod} \; q); \qquad \pi_2(0) \equiv \frac{q-1}{2} \; (\mathrm{mod} \; p).
 \]
 
 \item If $p$ is odd and $q$ is even, 
 \[
 \pi_1(0) \equiv \frac{p-1}{2} \; (\mathrm{mod} \; q); \qquad \pi_2(0) \equiv \frac{p+q-1}{2} \; (\mathrm{mod} \; p).
 \]
 
 \item If $p$ is even and $q$ is odd,
 \[
 \pi_1(0) \equiv \frac{p+q-1}{2} \; (\mathrm{mod} \; q); \qquad \pi_2(0) \equiv \frac{q-1}{2} \; (\mathrm{mod} \; p).
 \]
 
\end{enumerate}\end{lemma}
\begin{proof}
In \cite[Lemma 4.1]{Wu_cabling} it was shown that the projections maps satisfy the following relations
\[
\pi_1(0) \equiv - \frac{(p-1)(q-1)}{2} \; (\mathrm{mod} \; q); \qquad \pi_2(0) \equiv - \frac{(p-1)(q-1)}{2} \; (\mathrm{mod} \; p).
\]
The desired result now follows from a direct calculation.
\end{proof}

We may now prove Theorem~\ref{thm: intro_formula}.

\begin{proof}[Proof of Theorem~\ref{thm: intro_formula}]
We break the proof into three cases, depending on the parity of $p$ and $q$.

\textit{Case 1.} We first consider the case when $p$ and $q$ are both odd. We will give the proof for $\Vl$. The proof for $\Vu$ is identical. 
Per Lemma~\ref{lem: spinc_decom}, the projection of the $[0]$ $\spinc$ structure on $S^3_{pq}(K)$ to $L(p,q)$ has image the unique self-conjugate $\spinc$-structure, represented by
\[
\left[\frac{q-1}{2}\right] \; (\mathrm{mod} \; p).
\]
Similarly, for $q/p$-Dehn surgery with $q$ odd, the projection of the $[0]$ $\spinc$ structure on $S^3_{pq}(K)$ to $S^3_{q/p}(K)$ has image the unique self-conjugate $\spinc$-structure, namely
\[
\left[\frac{p-1}{2}\right] \; (\mathrm{mod} \; q).
\]
Now applying Lemma~\ref{lem: spinc_decom} and \eqref{eqn:connect_sum} to the case that both $p$ and $q$ are odd, we we see that the iota-complex $(\CFm(S^3_{pq}(K_{p,q}),[0]), \iota)$ is strongly equivalent to the iota-complex $\left(\CFm\left(S^3_{q/p}(K),\left[\frac{p-1}{2}\right]\right),\iota\right)$ shifted in grading by $d\left(L(p,q),\left[\frac{q-1}{2}\right]\right)$.  It therefore follows that 
\begin{equation}\label{equ:dl_1}
\dl(S^3_{pq}(K_{p,q}),[0])= \dl\left(S^3_{q/p}(K), \Big[ \frac{p-1}{2} \Big]\right) + d\left(L(p,q),\Big[\frac{q-1}{2} \Big]\right).
\end{equation}
Applying Theorem~\ref{thm:surgery_formula}, we get that 

\begin{equation}\label{equ:surgery_formula_1}
\dl(S^3_{pq}(K_{p,q}),[0])= d(L(pq,1),[0]) -2\Vl(K_{p,q})
\end{equation}
and also that
\begin{equation}\label{equ:surgery_formula_2}
\dl\left(S^3_{q/p}(K),\Big[\frac{p-1}{2} \Big]\right) = d\left(L(q,p),\Big[\frac{p-1}{2}\Big]\right) - 2\Vl(K).
\end{equation}
Combining Equations~\eqref{equ:surgery_formula_1} and \eqref{equ:surgery_formula_2} with Equation~\eqref{equ:dl_1} we get 
\begin{equation}\label{equ:combine}
 d(L(pq,1),[0]) -2\Vl(K_{p,q})=d\left(L(q,p),\Big[\frac{p-1}{2} \Big]\right) - 2\Vl(K) + d\left(L(p,q),\Big[\frac{q-1}{2} \Big]\right).
\end{equation}
Now by plugging in the unknot for $K$ in Equation~\ref{equ:combine}, we get
\begin{gather}\label{equ:unknot}
d(L(pq,1),[0]) -2\Vl(T_{p,q}) = d\left(L(q,p),\Big[\frac{p-1}{2} \Big]\right) +  d\left(L(p,q),\Big[\frac{q-1}{2} \Big]\right).
\end{gather}
Substituting Equation~\eqref{equ:unknot} into Equation~\eqref{equ:combine} and observing that $\Vl(T_{p,q})= V_0(T_{p,q})$ \cite[Section 7]{HM} we get the desired equality. The proof for $\Vu$ is similar.

\textit{Case 2.} Let us now consider the case when $p$ is odd and $q$ is even. We first consider $\Vl$. The two self-conjugate $\spinc$-structures on $S^3_{q/p}(K)$ are given by
\[
\left[\frac{p-1}{2}\right] \; (\mathrm{mod} \; q) \; \text{ and } \; \left[\frac{p+q-1}{2}\right] \; (\mathrm{mod} \; q).
\]
On the other hand, on $L(p,q)$ there is only one self-conjugate $\spinc$ structure, to wit
\[
\left[\frac{p+q-1}{2}\right] (\mathrm{mod} \; p).
\]
Now, in this case, Lemma~\ref{lem: spinc_decom} implies
\begin{equation}\label{equ:dl_3}
\dl(S^3_{pq}(K_{p,q}),[0])= \dl\left(S^3_{q/p}(K), \Big[ \frac{p-1}{2} \Big]\right) + d\left(L(p,q),\Big[\frac{p+q-1}{2} \Big]\right).
\end{equation}
Again applying Theorem~\ref{thm:surgery_formula}, we also have
\begin{equation}\label{equ:surgery_formula_4}
\dl(S^3_{pq}(K_{p,q}),[0])= d(L(pq,1),[0]) -2\Vl(K_{p,q})
\end{equation}
and
\begin{equation}\label{equ:surgery_formula_5}
\dl\left(S^3_{q/p}(K),\Big[\frac{p-1}{2} \Big]\right) = d\left(L(q,p),\Big[\frac{p-1}{2}\Big]\right) - 2\Vl(K).
\end{equation}
The rest of the proof is identical to that for the case when $p$ and $q$ are both odd. The proof for $\Vu$ is also similar.

\textit{Case 3.} We are left to consider the case when $p$ is even. In this case by Lemma~\ref{lem: spinc_decom} we have
\begin{equation}\label{eq:finalspinc}\dl(CF(S^3_{pq}(K_{p,q}),[0])) = \dl\left(S^3_{q/p}(K), \Big[ \frac{p+q-1}{2} \Big]\right) + d\left(L(p,q),\Big[\frac{q-1}{2} \Big]\right).\end{equation}
and from Theorem~\ref{thm:surgery_formula} we still have
\begin{equation}\label{eq:usual}\dl(CF(S^3_{pq}(K_{p,q}),[0])) = d(L(pq,1),[0]) - 2\Vl(K). \end{equation}
However, in this case our second consequence of Theorem~\ref{thm:surgery_formula} is of the form
\begin{equation}\label{equ:surgery_formula_6}
\dl\left(S^3_{q/p}(K),\Big[\frac{p+q-1}{2} \Big]\right) = d\left(S^3_{q/p}(K),\Big[\frac{p+q-1}{2}\Big]\right).
\end{equation}
Now let us choose $0 \leq s \leq q-1$ such that
\[
s \equiv \frac{p+q-1}{2} \; (\mathrm{mod} \; q).
\]
Applying the Ni-Wu formula \eqref{eq:NiWu} , we get
\begin{equation}\label{equ:surgery_formula_7}
d\left(S^3_{q/p}(K),\Big[\frac{p+q-1}{2}\Big]\right) = d\left(L(q,p),\Big[\frac{p+q-1}{2}\Big]\right) - 2 \; \mathrm{max}\left \{ V_{\left \lfloor \frac{s}{p} \right \rfloor }(K), V_{\left \lfloor \frac{p+q -1-s}{p} \right \rfloor }(K) \right \}.
\end{equation}
Combining Equations~\eqref{eq:usual}, \eqref{eq:finalspinc}, \eqref{equ:surgery_formula_6} and \eqref{equ:surgery_formula_7} and following the similar steps as above gives us the desired equality for $\Vl$.

As $\Vl$ and $\Vu$ have significantly different expressions in this final case, we also comment briefly on the argument for $\Vu$. The analog of Equation~\eqref{equ:surgery_formula_6} is
\begin{equation}\label{equ:surgery_formula_8}
\du\left(S^3_{q/p}(K),\Big[\frac{p+q-1}{2} \Big]\right) = d\left(L(q,p),\Big[\frac{p+q-1}{2}\Big]\right).
\end{equation}
Therefore following similar steps as for the $\Vl$ case, we get
\[
\Vu(K_{p,q}) = V_0(T_{p.q}).
\]
This completes the proof.
\end{proof}
The proof of Corollary~\ref{cor:intro_ite_cable} now follows.
\begin{proof}[Proof of Corollary~\ref{cor:intro_ite_cable}]
We may apply Theorem~\ref{thm: intro_formula} recursively.
\end{proof}

\subsection{Unknotting number bound}
In this subsection, we produce a bound for the unknotting number, using Theorem~\ref{thm: intro_formula}.

\begin{proof}[Proof of Theorem~\ref{thm: intro_unknotting}]
In \cite[Theorem 1.7]{JZ:clasp} Juh{\'a}sz-Zemke proved a bound for the slice genus using the involutive concordance invariants. Specifically, they showed
\begin{equation}
-\left\lceil \frac{g_{4}(K) + 1}{2} \right\rceil \leq \Vu(K) \leq \Vl(K) \leq \left\lceil \frac{g_{4}(K) + 1}{2} \right\rceil.
\end{equation}
Now observe that
\[
g_4(K) \leq u(K).
\]
Hence, we immediately get the desired inequality by replacing $K$ with $K_{p,q}$ and applying Theorem~\ref{thm: intro_formula} for $p$ odd.
\end{proof}

\section{Examples} \label{sec:examples}
We now discuss an example for which our cabling formula improves on other bounds from knot Floer homology.

\begin{example} \label{example:beats-previous}
We consider the knot $K:= -2 T_{6,7} \# T_{6,13}$. It was shown in \cite[Section 4]{hendricks2020surgery} that the knot Floer chain complex of $K$ splits into equivariant summands one of which is the tensor product of $\mathbb F[U,U^{-1}]$ with the complex $C$ in Figure \ref{fig:cfk} in such a way that $\Vl(K)$ and $\Vu(K)$ are equal to the involutive concordance invariants of the complex in Figure \ref{fig:cfk} with the action of $\iota$ shown; in the language of the literature, the two complexes are \emph{locally equivalent}. Indeed, one may compute $\Vl$ and $\Vu$ of the knot $K$ by computing $\dl$ and $\du$ of a suitable subcomplex of $C \otimes \mathbb F[U,U^{-1}]$; for more on this computation, see \cite[Section 6]{HM} and \cite[Introduction]{hendricks2020surgery}.

\begin{figure}[h!]
\center
\includegraphics[scale=0.7]{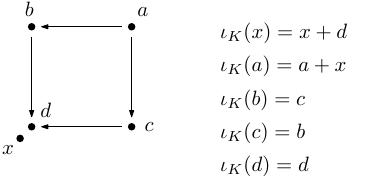}
\caption{The knot Floer chain complex of the knot $K$, up to local equivalence. The arrows are of length three. The action of $\iota$ is shown in the right}\label{fig:cfk}
\end{figure}

In particular, computation shows that $\Vl(K)=3$. Now the unknotting number bound for $K_{3,2}$  from Theorem~\ref{thm: intro_unknotting} gives
\[
u(K_{3,2}) \geq 2 \Vl(K) + 2 - 2 = 6 
\]
Note that slice genus, and hence unknotting number, bounds for $K_{3,2}$ coming from the Ozsv{\'a}th-Szab{\'o} $\tau$-invariant, from the $V_0$-invariant, and from the $\nu^{+}$ invariant \cite{Wu_cabling} are all 1. Finally, observe that the unknotting number bound from \cite{HLP:cables} shows that $u(K_{3,2}) \geq 3$. 

A similar argument applies to the general situation of the $(n,2)$-cables of the knots $K_n = -2T_{2n,2n+1}\#T_{2n,4n+1}$ for $n$ odd considered in \cite{HHSZ:quotient}, which have chain complexes locally equivalent to the analog of the complex in Figure~\ref{fig:cfk} with length $n$ arrows in the box, and have $\Vl(K_n) = n$.
\end{example}

\bibliographystyle{amsalpha}
\bibliography{bib}

\end{document}